\documentclass[submission]{FPSAClike}

\usepackage[utf8]{inputenc}
\usepackage{mathtools}

\newtheorem{thm}{Theorem}

\newtheorem{cor}{Corollary}
\newtheorem{pro}{Proposition}
\newtheorem{con}{Conjecture}

\title[Fighting Fish: enumerative properties]{Fighting Fish:
  enumerative properties}

\author[Duchi, Guerrini, Rinaldi and Schaeffer]{Enrica Duchi
\addressmark{1}, \and Veronica Guerrini\addressmark{2}, \and Simone Rinaldi\addressmark{2},\\ \and Gilles Schaeffer\addressmark{3}}

\address{\addressmark{1}IRIT, Universit\'e Paris Diderot, Paris, France\\ 
\addressmark{2}Università di Siena, Siena, Italy\\
\addressmark{3}LIX, CNRS, \'Ecole Polytechnique, Palaiseau, France\\
}

\received{\today}

\abstract{Fighting fish were very recently introduced by the authors
  as combinatorial structures made of square tiles that form two
  dimensional branching surfaces. A main feature of these fighting fish is
  that the area of uniform random fish of size $n$ scales like
  $n^{5/4}$ as opposed to the typical $n^{3/2}$ area behavior of the
  staircase or direct convex polyominoes that they generalize.

  In this extended abstract we concentrate on enumerative properties
  of fighting fish: in particular we provide a new decomposition and
  we show that the number of fighting fish with $i$ left lower free
  edges and $j$ right lower free edges is equal to
  \begin{equation*}
  \frac{(2i+j-2)!(2j+i-2)!}{i!j!(2i-1)!(2j-1)!}.
  \end{equation*}
  These numbers are known to count rooted planar non-separable maps
  with $i+1$ vertices and $j+1$ faces, or two-stack-sortable
  permutations with respect to ascending and descending runs, or left
  ternary trees with respect to vertices with even and odd abscissa.
  However we have been unable until now to provide any explicit bijection
  between our fish and such structures. Instead we provide new refined
  generating series for left ternary trees to prove further
  equidistribution results.  }


\usepackage[backend=bibtex]{biblatex}
\begin{document}

\maketitle

\section{Introduction}

In a recent paper \cite{fighting-fish} we introduced a new family of
combinatorial structures which we call {\em fighting fish} since they
are inspired by the aquatic creatures known under the same name (see
\href{https://fr.wikipedia.org/wiki/Combattant}{the wikipedia page on Betta Splendens fish}).  The easiest
description of fighting fish is that they are built by gluing together
unit squares of cloth along their edges in a directed way that
generalize the iterative construction of directed convex polyominoes
\cite{tony}.

\begin{figure}
\begin{center}
\includegraphics[scale=0.8]{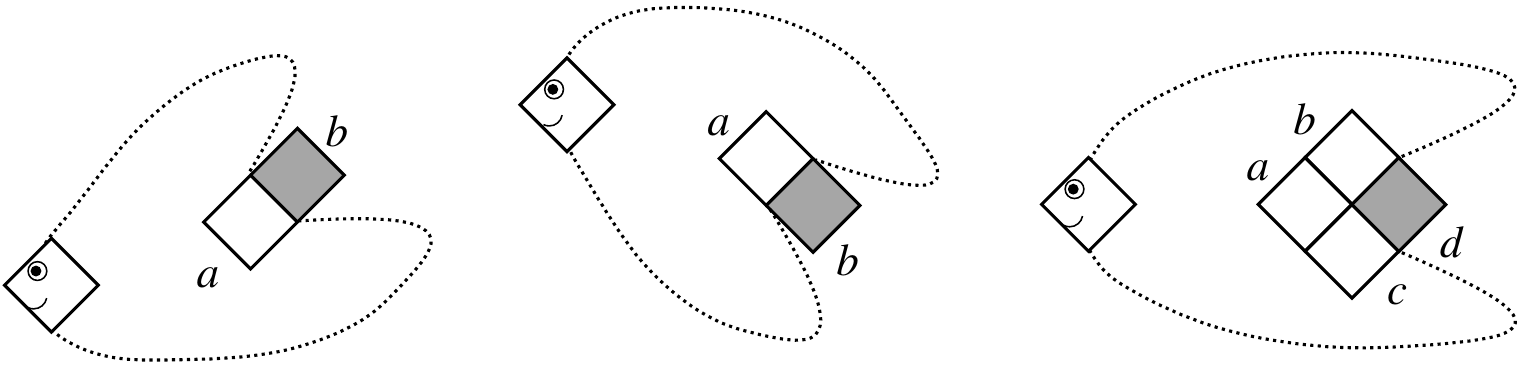}
\end{center}
\caption{The three ways to grow a fish by adding a cell.}
\label{fig:IterativeConstruction}
\end{figure}

More precisely,  we consider 45 degree tilted
unit squares which we call \emph{cells}, and
we call the four edges of these cells \emph{left upper edge},
\emph{left lower edge}, \emph{right upper edge} and \emph{right
lower  edge} respectively. We intend to glue cells along edges, and we call
\emph{free} an edge of a cell which it is not glued to an edge of
another cell. All fighting fish are then obtained from an initial cell
called the \emph{head} by attaching cells one by one in one of the
three following ways: (see Figure~\ref{fig:IterativeConstruction})
\begin{itemize}
\item Let $a$ be a cell already in the fish whose right upper
  edge is free; then glue the left lower edge of a new cell $b$ to the
  right upper edge of $a$.

\item Let $a$ be a cell already in the fish whose right lower
  edge is free; then glue the left upper edge of a new cell $b$ to the
  right lower edge of $a$.

\item Let $a$, $b$ and $c$ be three cells already in the fish and such
  that $b$ (resp. $c$) has its left lower (resp. upper) edge glued to
  the right upper (resp. lower) side of $a$, and $b$ (resp. $c$) has
  its right lower (resp. right upper) edge free; then simultaneously
  glue the left upper and lower edges of a new cell $d$ respectively to
  the right lower edge of $b$ and to the right upper edge of $c$.
\end{itemize}
While this description is iterative we are interested in the objects
that are produced, independently of the order in which cells are
added: a \emph{fighting fish} is a collection of cells glued together
edge by edge that \emph{can} be obtained by the iterative process
above. The \emph{head} of the fighting fish is the only cell with two
free left edges, its \emph{nose} is the leftmost point of the head; a
\emph{final cell} is a cell with two free right edges, and the corresponding \emph{tail} is its rightmost point; the \emph{fin} is the
path that starts from the nose of the fish, follows its border
counterclockwise, and ends at the first tail it meets (see Figure
\ref{fig:fish1}(a)).

The \emph{size} of a fighting fish is the number of lower free  edges
(which is easily seen to be equal to the number of upper free
edges). Moreover, the \emph{left size} (resp. \emph{right size}) of a
fighting fish is its number of left lower free edges (resp. right
lower free edges). Clearly, the left and right size of a fish sum to its size.
The \emph{area} of a fighting fish is the number of its cells.

\medskip

Examples of fighting fish are parallelogram polyominoes (aka staircase
polyominoes), directed convex polyominoes, and more generally simply
connected directed polyominoes in the sense of \cite{tony}. However,
one should stress the fact that fighting fish are not necessarily
polyominoes because cells can be adjacent without being glued together
and more generally cells are not constrained to fit in the plane and
can cover each other, as illustrated by Figure~\ref{fig:fish1}(a)
and~\ref{fig:fish1}(d).  The two smallest fighting fish which are not
polyominoes have size 5 and area 4: as illustrated by
Figure~\ref{fig:fish1}(c), they are obtained by gluing a square $a$ to
the upper right edge of the head, a square $b$ to the right lower edge
of the head and either a square $c$ to the right lower edge of $a$, or
a square $d$ to the upper right edge of $b$.  The smallest fighting
fish not fitting in the plane has size 6 and area 5, it is obtained by
gluing both $c$ to $a$ and $d$ to $b$: in the natural projection of
this fighting fish onto the plane, squares $c$ and $d$ have the same
image. Observe that we do not specify whether $c$ is above or below
$d$; rather we consider that the surface has a branch point at vertex
$c\cap d$ (see Figure~\ref{fig:fish1}(d)).

\begin{figure}
\begin{center}
\includegraphics[scale=0.7]{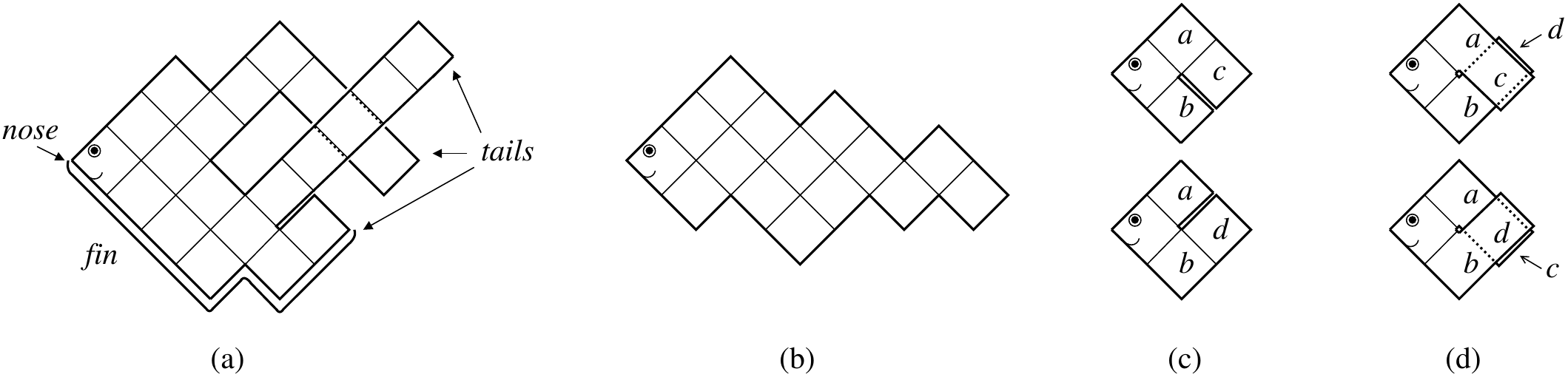}
\end{center}
\caption{(a) A fighting fish which is not a polyomino; (b) A
  parallelogram polyomino; (c) The two fighting fish with area 4 that
  are not polyominoes; (d) Two different representations of the unique
  fighting fish with area 5 not fitting in the plane.}
\label{fig:fish1}
\end{figure}

\medskip
In \cite{fighting-fish} we obtained the generating series of fighting
fish using essentially Temperley's approach, that is a decomposition
in vertical slices. This allowed us to prove:
\begin{thm}[\cite{fighting-fish}]\label{thm:fighting-fish}
The number of fighting fish with $n+1$ lower free edges is 
\begin{equation}
\frac{2}{(n+1)(2n+1)}{3n\choose n}
\end{equation}
\end{thm} 
We showed moreover that the average area of fighting fish of size $n$
is of order $n^{5/4}$.  This behavior suggests that, although fighting
fish are natural generalizations of directed convex polyominoes, they
belong to a different universality class: indeed the order of
magnitude of the area of most classes of convex polyominoes is rather
$n^{3/2}$ \cite{richard}.

In the present extended abstract we explore further the remarkable
enumerative properties of fighting fish. We propose in
Section~\ref{sec:deco} a new decomposition that extends to fighting
fish the classical wasp-waist decomposition of polyominoes
\cite{MBMinTony}.  Using the resulting equation we compute in
Section~\ref{sec:gfP1} the generating series of fighting fish with
respect to the numbers of left and right lower free edges, fin length and
number of tails, and use the resulting explicit parametrization to
prove the following bivariate extension of
Theorem~\ref{thm:fighting-fish}:
\begin{thm}\label{thm:count}
The number of fighting fish with $i$ left lower free edges and $j$ right
lower free edges is
\begin{equation}\label{for:for1}
\frac{(2i+j-2)!(2j+i-2)!}{i!j!(2i-1)!(2j-1)!}
=\frac1{ij}{2i+j-2\choose j-1}{2j+i-2\choose i-1}.
\end{equation}
\end{thm}
We also discuss in Section~\ref{sec:marked} several remarkable
relations between fighting fish with marked points of various
types. In particular we prove:
\begin{thm}\label{thm:marked}
The number of fighting fish with $i$ left lower free and $j$ right lower free
edges with a marked tail is
\begin{equation}\label{for:for2}
\frac{(2i+2j-3)}{(2i-1)(2j-1)}{2i+j-3\choose j-1}{2j+i-3\choose i-1}
\end{equation}
\end{thm}
All these results confirm the apparently close relation of fighting
fish to the well studied combinatorial structures known as \emph{non
  separable planar maps} \cite{B}, \emph{two stack sortable
  permutations} \cite{west,Zeilb,bona}, and \emph{left ternary trees}
\cite{DDP,JS}.  The closest link appears to be between fighting fish
and left ternary trees, that is, ternary trees whose vertices all have
non negative abscissa in the natural embedding \cite{kuba}. We prove in
Section~\ref{sec:ternary} the following theorem, which was conjectured
in \cite{fighting-fish}:
\begin{thm}\label{thm:fin-core}
The number of fighting fish with size $n$ and fin length $k$ is equal
to the number of left ternary trees with $n$ nodes, $k$ of which are
accessible from the root using only left and middle branches.
\end{thm}
We prove this theorem by an independent computation of the generating
series of left ternary trees with respect to $n$ and $k$ (see
Theorem~\ref{thm:generalizedKuba}), building on Di Francesco's method
\cite{diFrancesco2} for counting positively labeled trees.  As
discussed in Section~\ref{sec:ternary} we conjecture that
Theorem~\ref{thm:fin-core} extends to take into account the left and
right size and the number of tails but we have only been able to prove
this bijectively in the case of fighting fish with at most two tails,
and in the case of fighting fish with $h$ tails but at most $h+2$
lower edges that are not in the fin.

\section{A wasp-waist decomposition}\label{sec:deco}

\begin{thm}\label{thm:wasp-waist}
Let $P$ be a fighting fish. Then exactly one of the following cases (A), (B1), (B2), (C1), (C2) or (C3) occurs:
\begin{description}
\item{(A)} $P$ consists of a single cell;
\item{(B)} $P$ is obtained from a smaller size fighting fish $P_1$:
\begin{description}
\item{(B1)} by gluing the right lower edge of a new cell to the upper
  left edge of the head of $P_1$ (Figure \ref{fig:waist_fish}~(B1));
\item{(B2)} by gluing every left edge of the fin of $P_1$ to the upper
  right edge of a new cell, and gluing the right lower edge and the
  upper left edge of all pairs of adjacent new cells (Figure
  \ref{fig:waist_fish}~(B2));
\end{description}
\item{(C)} $P$ is obtained from two smaller size fighting fish, $P_1$ and $P_2$:
\begin{description}
\item{(C1)} by performing to $P_1$ the operation described in~(B2) and
  then gluing the upper left edge of the head of $P_2$ to the last
  edge of the fin of $P_1$ (Figure \ref{fig:waist_fish}~(C1));
\item{(C2)} by choosing a right edge $r$ on the fin of $P_1$ (last
  edge of the fin excluded) and gluing every left edge preceding $r$
  on the fin to the upper right edge of a new cell and, as above,
  gluing the right lower edge and the upper left edge of every pair of
  adjacent new cells; Then, gluing the upper left edge of the head of
  $P_2$ to $r$ (Figure \ref{fig:waist_fish}~(C2));
\item{(C3)} by choosing a left edge $\ell$ on the fin of $P_1$ and gluing
  every left edge of the fish fin preceding $\ell$ (included) to the
  upper right edge of a new cell and, as above, gluing the right lower
  edge and the upper left edge of every pair of adjacent new
  cells; Then, gluing the upper left edge of the head of $P_2$ to the
  right lower edge of the cell glued to $\ell$ (Figure
  \ref{fig:waist_fish}~(C3)).
\end{description}
\end{description}
Moreover each of the previous operations, when applied to arbitrary fighting fish $P_1$ and if necessary $P_2$, produces valid a fighting fish. 
\end{thm}
\begin{figure}
\centering
\includegraphics[width=1\textwidth]{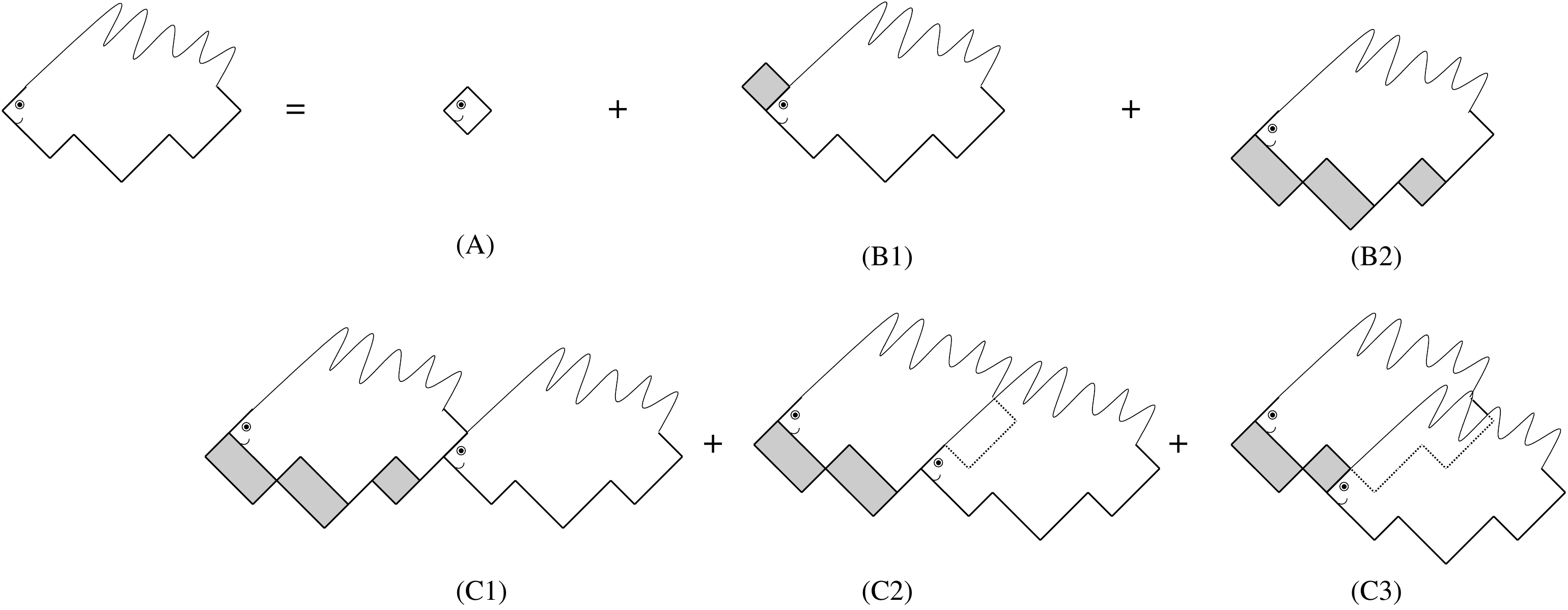}
\caption{The wasp-waist decomposition.}
\label{fig:waist_fish}
\end{figure}

Observe that Cases (A), (B1) and (B2) could have been alternatively
considered as degenerate cases of Case (C1) where $P_1$ or $P_2$ would
be allowed to be empty. Staircase polyominoes are exactly the
fighting fish obtained using only Cases (A), (B1), (B2) and (C1).

\begin{proof} Omitted (Appendix~\ref{ap:deco}).
\end{proof}

Let $P(t,y,a,b;u)=\sum_{P} t^{\mathrm{size} (P)-1} y^{\mathrm{tails}
  (P)-1} a^{\mathrm{rsize} (P)-1} b^{\mathrm{lsize} (P)-1}
u^{\mathrm{fin} (P)-1}$ denote the generating series of fighting fish
with variables $t,y,a,b$ and $u$ respectively marking the size, the
number of tails, the right size, the left size, the fin length, all
decreased by one.

\begin{cor}\label{cor:gf} 
The generating series $P(u)\equiv P(t,y,a,b;u)$
of 
fighting fish satisfies
the equation
\begin{equation}\label{eqn:eq0}
P(u)=tu(1+aP(u))(1+bP(u))+ytabuP(u)\frac{P(1)-P(u)}{1-u}.
\end{equation}
\end{cor}
\begin{proof} This is a direct consequence of the previous theorem, 
details are omitted (Appendix~\ref{ap:deco}).
\end{proof}

\section{Enumerative results for fish}\label{sec:count}
\subsection{The algebraic solution of the functional equation}\label{sec:gfP1}
The equation satisfied by fighting fish is a combinatorially funded
polynomial equation with one catalytic variable: this class of
equations was thoroughly studied by Bousquet-M\'elou and Jehanne
\cite{MBMA} who proved that they have algebraic solutions.
\begin{thm}\label{thm:gfP1}
Let $B\equiv B(t;y,a,b)$ denote the unique power series solution of the equation
\begin{equation}\label{eqn:eq6U}
B=t\left(1+y\frac{abB^2}{1-abB^2}\right)^2(1+aB)(1+bB).
\end{equation}
Then the generating series $P(1)\equiv P(t;y,a,b,1)$ of fighting fish can be expressed as 
\begin{equation}\label{eqn:eq8U}
P(1)=B-\frac{yabB^3(1+aB)(1+bB)}{(1-abB^2)^2}.
\end{equation}
\end{thm}
This theorem easily implies Theorem~\ref{thm:count} using Lagrange
inversion (Appendix~\ref{ap:lagrange}).
\begin{proof}[Proof of Theorem~\ref{thm:gfP1}]
Our proof follows closely the approach of \cite{MBMA}, so we omit the
details (Appendix~\ref{ap:gf}) and only present the strategy: Rewrite
Equation~\eqref{eqn:eq0} as
\begin{equation}\label{eqn:eq1}
(u-1)P(u)=tu(u-1)(1+aP(u))(1+bP(u))+ytuabP(u)(P(u)-P(1)).
\end{equation}
and take the derivative with respect to $u$:
\begin{align*}
\MoveEqLeft
P(u)-t(2u-1)(1+aP(u))(1+bP(u))-ytabP(u)(P(u)-P(1)) \\
&=-\frac{\partial}{\partial u}P(u)\cdot \left(u-1-tu(u-1)(a+b+2abP(u))-ytuab(2P(u)-P(1))\right)
\end{align*}
Now there clearly exists a unique power series $U$ that cancels the
second factor in the right hand side of the previous equation: $U$ is
the unique power series root of the equation
\begin{equation}\label{eqn:eq2U}
U-1=tU(U-1)(a+b+2abP(U))+ytabU(2P(U)-P(1)).
\end{equation}
Since $U$ must also cancels the left hand side,
\begin{equation}\label{eqn:eq3U}
P(U)=t(2U-1)(1+aP(U))(1+bP(U))+ytabP(U)(P(U)-P(1)), 
\end{equation}
and, for $u=U$, Equation~\eqref{eqn:eq1} reads
\begin{equation}\label{eqn:eq1U}
(U-1)P(U)=tU(U-1)(1+aP(U))(1+bP(U))+ytUabP(U)(P(U)-P(1)).
\end{equation}
Solving the resulting system of 3 equations for the three unknown 
$U$, $P(U)$ and $P(1)$ yields the theorem, with $P(U)=B$.
\end{proof}

The full series $P(u)$ is clearly algebraic of degree at most 2 over
$\mathbb{Q}(u,B)$, but it admits in fact a parametrization directly
extending the one of the theorem. 
\begin{cor}\label{cor:Pu}
Let $B(u)$ be the unique power series solution of the equation:
\begin{equation}
B(u)=tu\left(1+aB(u)+yaB(u)\frac{bB(1+aB)}{1-abB^2}\right)
\left(1+bB(u)+ybB(u)\frac{aB(1+bB)}{1-abB^2}\right)
\end{equation}
then
\begin{equation*}
P(u)=B(u)
-yabB(u)^2B\frac{(1+aB)(1+bB)(1-abB^2+yabB^2)}
{(1-abB^2)^2(1-abB(u)B+yabB(u)B)}.
\end{equation*}
\end{cor}

\subsection{Fighting fish with marked points}\label{sec:marked}
\begin{figure}
\centering
\includegraphics[width=0.8\textwidth]{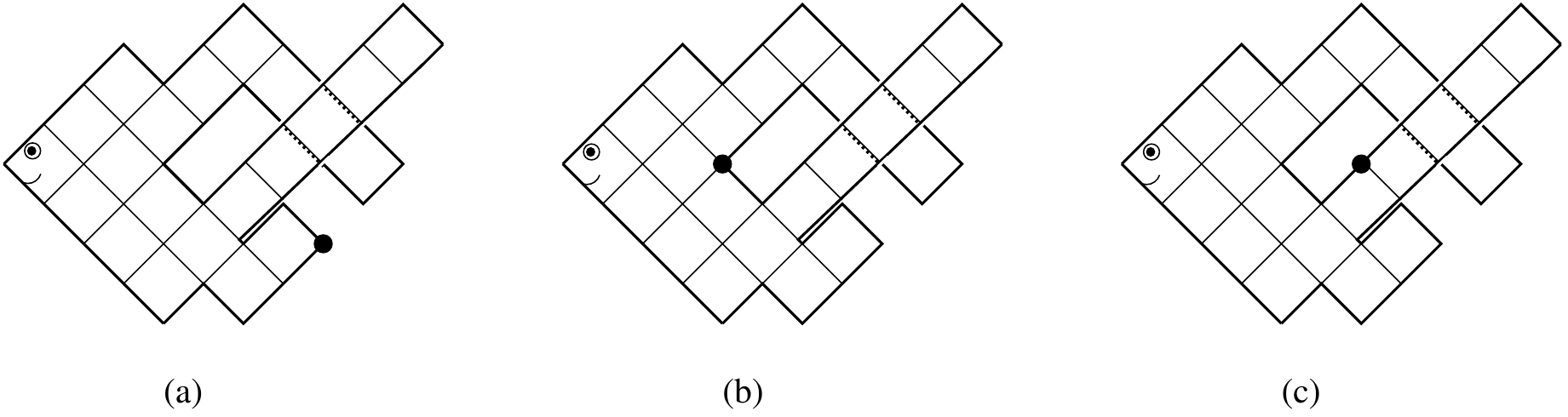}
\caption{Fish with marked points: (a) a tail, (b) a branch point, (c) an upper flat point.}
\label{fig:markedpoints}
\end{figure}

Let $P^<$ denote the generating series of fish with a marked branched
point, $P^>$ the generating series of fish with a marked tail,
$2P^{-}$ the generating series of fish with a marked \emph{flat} point, that
is, a marked point which is neither the head nor a tail nor a branched
point (observe that each fish has the same number of upper and lower
flat points, hence the factor 2). The generating series of fighting
fish with a marked point is then:
\begin{equation}
P(1)+2P^{-}+P^>+P^<=2\frac{t\partial}{\partial t}(tP(1)).
\end{equation}

From the fact that there is always one more tail than branch point we have
\begin{equation}
P(1)+P^<=P^>
\end{equation}
so that we also have 
\begin{equation}
P^-+P^>=P(1)+P^<+P^-=\frac{t\partial}{\partial t}(tP(1)).
\end{equation}
Fighting fish with a marked tail can also be counted thanks to the variable $y$:
\begin{equation}
P^<=\frac{y\partial}{\partial y}P(1), \quad\textrm{ or }\quad 
P^>=\frac{y\partial}{\partial y}(yP(1)).
\end{equation}

Observe that derivating Equation~\eqref{eqn:eq1} with respect to $y$
instead of $u$ yields the same coefficient for the derivative of
$P(u)$, which cancels for $u=U$. This simplification leads to 
the remarkable relations:
\begin{equation}\label{eq:markedPU}
P^<=y\frac{\partial}{\partial y} P(1)=P(U)-P(1),
\qquad\textrm{ and }
P^>=P(U).
\end{equation}
This relation allows to use bivariate Lagrange inversion on the
parametrization $P(U)=B$ in Theorem~\ref{thm:gfP1} to prove
Theorem~\ref{thm:marked}, we omit the details (Appendix~\ref{ap:lagrange}).

Similarly derivating Equation~\eqref{eqn:eq1} with respect to $t$ and taking $u=U$ yields:
\begin{equation}\label{eq:eqU}
U=\frac1{1-V} \quad\textrm{ where }\quad V=ytab\frac{t\partial}{\partial t} P(1)=ytab(P^-+P^<).
\end{equation}

Equations~\eqref{eq:markedPU} and~\eqref{eq:eqU} admit direct combinatorial
interpretations (Appendix~\ref{ap:bijections}).

\section{Fighting fish and left ternary trees: the fin/core relation}\label{sec:ternary}

A \emph{ternary tree} is a finite tree which is either empty or
contains a root and three disjoint ternary trees called the left,
middle and right subtrees of the root. Given a initial root label $j$,
a ternary tree can be naturally embedded in the plane in a
deterministic way: the root has abscissa $j$ and the left
(resp. middle, right) child of a node with abscissa $i\in\mathbb{Z}$
has abscissa $i-1$ (resp, $i$, $i+1$). A \emph{$j$-positive tree} is a
is a ternary tree whose nodes all have non positive abscissa;
$0$-positive trees were first introduced in the literature with the
name \emph{left ternary tree} \cite{DDP,JS} (in order to be coherent
with these works one should orient the abscissa axis toward the left).

It is known that the number of left ternary trees with $i$ nodes at
even position and $j$ nodes at odd position is given by
Formula~\eqref{for:for1} \cite{DDP,JS}. In order to refine this result
we introduce the following new parameters on left ternary trees:
\begin{itemize}
\item Let the \emph{core} of a ternary tree $T$ be the largest
  subtree including the root of $T$ and consisting only of left and
  middle edges. 
\item Let a \emph{right branch} of a ternary tree be a maximal
  sequence of right edges.
\end{itemize}
In order to prove Theorem~\ref{thm:fin-core} we compute the generating
series of $j$-positive trees according to the number of nodes and
nodes in the core.

\subsection{A refined enumeration of $j$-positive trees}
\textbf{In this section we implicitly take $y=a=b=1$ in all generating
  series.}

Let $T$, $B$ and $X$ be the unique formal power series solutions of
\begin{align*}
T=1+tT^3,\quad\textrm{ and }\quad B=tT^2
\quad\textrm{ and }\quad 
X=B(1+X+X^2).
\end{align*}
Observe that $B$ coincide with the series $B(t;1,1,1,1)$ of the
previous sections and that,
\begin{align*}
T=\frac1{1-B}=\frac{1+X+X^2}{1+X^2},\quad\textrm{ and }\quad
T-1=BT=\frac{X}{1+X^2}, \quad\textrm{ and }\quad
B=\frac{X}{1+X+X^2}.
\end{align*}
Building on Di Francesco's educated guess and check approach
\cite{diFrancesco2}, Kuba obtained a formula for $j$-positive trees
reads:
\begin{thm}[\cite{kuba,diFrancesco2}] \label{thm:kuba}
The generating series $T_j=T_j(t;1,1,1,1)$ is given for all $j\geq0$
by the explicit expression:
\begin{equation*}
T_j=T\frac{(1-X^{j+5})(1-X^{j+2})}{(1-X^{j+4})(1-X^{j+3})}.
\end{equation*}
\end{thm}
Define moreover 
\begin{align*}
T(u)&=1+tuT(u)^2T\quad\textrm{ and }\quad
B(u)=tuT(u)^2
\end{align*}
so that 
\begin{align*}
T(1)=T,\quad\textrm{ and }\quad T(u)=1+B(u)T,
 \quad\textrm{ and }\quad B(u)=(T(u)-1)(1-B)
\end{align*}
Observe that this $B(u)$ coincide with the series $B(t;1,1,1,u)$ of
the previous sections, so that the generating series of fighting fish
according to the size and the fin length, given by
Corollary~\ref{cor:Pu}, can be written as (recall that here $y=a=b=1$)
\begin{align}\label{eqn:eqPuTu}
1+P(u)&=1+B(u)-B(u)^2\frac{B}{(1-B)^2}
=T(u)(1+B)-T(u)^2B.
\end{align}

\begin{thm} \label{thm:generalizedKuba}
The generating series $T_j(u)\equiv T_j(t;1,1,1,u)$ is given for $j\geq-1$ by
\begin{align*} 
T_j(u)&=T(u)\frac{H_j(u)}{H_{j-1}(u)}\frac{1-X^{j+2}}{1-X^{j+3}}
\end{align*}
where for all $j\geq-2$, 
\begin{align*}
H_j(u)&=(1-X^{j+1})XT(u)-(1+X)(1-X^{j+2}).
\end{align*}
\end{thm}
\begin{cor}\label{cor:basic}
The number of left ternary trees with $n$ vertices, $k$ of which belong to the core, is equal to the number of fighting fish of size $n$ with fin length $k$.
\end{cor}
\begin{proof}[Proof of Corollary~\ref{cor:basic}]
This is a simple computation:
\begin{align*}
T_0(u)&=T(u)\frac{H_0(u)}{H_{-1}(u)}\frac{1-X^{2}}{1-X^{3}}.
\end{align*}
By definition $H_{-1}=-(1+X)(1-X)$ and $H_{-2}(u)=(X-1)T(u)$,
so that
\begin{align*}
T_0(u)&=T(u)\frac{(1-X)XT(u)-(1+X)(1-X^2)}{-(1-X^2)}\frac{1-X^{2}}{1-X^{3}}
=-T(u)^2B+T(u)(1+B).
\end{align*}
which coincide with Equation~\eqref{eqn:eqPuTu}.
\end{proof}

\begin{proof}[Proof of Theorem~\ref{thm:generalizedKuba}]
In order to prove the theorem it is sufficient to show that the series
given by the explicit expression satisfies for all $j\geq -1$ the
equation:
\begin{equation}\label{eq:rectrees}
T_{j}(u)=1+tuT_{j+1}(u)T_{j}(u)T_{j-1}
\end{equation}
where $T_j$ is given by Theorem~\ref{thm:kuba}, with the convention
that $T_{-2}=0$: indeed the system of Equations~\eqref{eq:rectrees}
clearly admits the generating series of $j$-positive ternary trees as
unique power series solutions.
The case $j=-1$ is immediate:
\begin{equation*}
T_{-1}(u)=T(u)\frac{H_{-1}(u)}{H_{-2}(u)}\frac{1-X}{1-X^2}=1.
\end{equation*}
Let now $j\geq0$, then the right hand side of Equation~\eqref{eq:rectrees} reads
\begin{align*}
1+&tuT(u)^2T\cdot\frac{H_{j+1}(u)}{H_{j-1}(u)}\cdot
\frac{1-X^{j+4}}{1-X^{j+3}}\cdot\frac{1-X^{j+1}}{1-X^{j+4}}\\&
=\frac{H_{j-1}(u)(1-X^{j+3})+(T(u)-1)H_{j+1}(u)(1-X^{j+1})}{H_{j-1}(u)(1-X^{j+3})}
\end{align*}
and we want to show that this is equal to 
\begin{align*}
\frac{T(u)H_{j}(u)(1-X^{j+2})}{H_{j-1}(u)(1-X^{j+3})}.
\end{align*}
Now
\begin{align*}
H_{j-1}(u)(1-X^{j+3})
&=(1-X^{j+3})(1-X^{j})XT(u)-(1+X)(1-X^{j+1})(1-X^{j+3}),\\
-H_{j+1}(u)(1-X^{j+1})
&=-(1-X^{j+1})(1-X^{j+2})XT(u)+(1+X)(1-X^{j+3})(1-X^{j+1}),\\
T(u)H_{j+1}(u)(1-X^{j+1})&=(1-X^{j+1})(1-X^{j+2})XT(u)^2-(1+X)(1-X^{j+3})(1-X^{j+1})T(u),
\end{align*}
while
\begin{align*}
T(u)H_j(u)(1-X^{j+2})
&=(1-X^{j+2})(1-X^{j+1})XT(u)^2-(1+X)(1-X^{j+2})^2T(u).
\end{align*}
The coefficients of $T(u)^2$ and $T(u)^0$ are clearly matching. Upon
expanding all contributions to the coefficient of $T(u)$ in
power of $X$, the various terms are directly seen to match as well.
\end{proof}

\subsection{A refined conjecture}
In view of Theorem~\ref{thm:count} and Theorem~\ref{thm:fin-core}, it
is natural to look for a common generalization. Indeed one can even
take the number of tails into account:
\begin{con}\label{fin-core}
The number of fighting fish with size $n$, fin length $k$, having $h$
tails, with $i$ left lower free edges and $j$ right lower free edges
is equal to the number of left ternary trees with $n$ nodes, core size
$k$, having $h$ right branches, with $i+1$ non root nodes with even
abscissa and $j$ nodes with odd abscissa.
\end{con}
This conjecture naturally calls for a bijective proof, however we have
been unable to provide such a proof, except in two specific cases:
\begin{itemize}
\item The case of left ternary trees with at most one right branch, which are in bijections with fighting fish with at most two tails for all values of $n$, $k$, $i$ and $j$.
\item The case of left ternary trees with $h$ right branches and at
  most $h+2$ vertices, which are in bijection with fighting fish with
  $h$ tails and $h+2$ lower edges that are not in the fin.
\end{itemize}

\acknowledgements{The first and last author are grateful to the
  university of Siena where part of this work was done.}

\printbibliography

\newpage
\appendix

\begin{figure}
\centering
\includegraphics[width=0.4\textwidth]{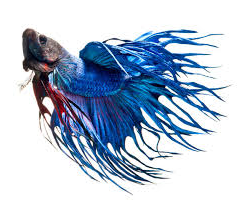}
\caption{A real fighting fish.}
\label{fig:ff}
\end{figure}

\section{Proof of Theorem~\ref{thm:wasp-waist} and Corollary~\ref{cor:gf}}\label{ap:deco}
\begin{proof}[Proof of the theorem]
The operations described in Theorem~\ref{thm:wasp-waist} produce valid
fighting fish: indeed given incremental growths of $P_1$ and $P_2$ we
obtain a valid incremental growth of $P$ upon starting from the new
head, growing the head of $P_1$ interleaving the next steps of the
growth of $P_1$ with insertions of the new cells: each new cell is to
be inserted just before the fin cell it will be attached to; when this
is done, the head of the fish $P_2$ can be attached on $P_1$ and the
rest of $P_2$ growth from there.

It thus remains  to show that every fighting fish of size greater
that $2$ can be uniquely obtained by applying one of the operations
$(B)$ or $(C)$ to fish of smaller size.

In order to prove the result let us describe how to decompose a fish
$P$ which is not reduced to a cell. In order to do this we need two
further definition: First let us call \emph{cut edge} of $P$ any
common side $e$ of two cells of $P$ such that cutting $P$ along $e$
yields two connected components. Second let the set of \emph{fin
  cells} of $P$ be the set of cells incident to a left edge of the
fin: the head of $P$ is always a fin cell and the other fin cells have
non-free left upper sides (since their left lower sides are free and
they must be attached by a left side).

Now the decomposition is as follows: 
\begin{itemize}
\item First mark the head of $P$ as \emph{removable} and consider the
  other fin cells iteratively from left to right: mark them as
  \emph{removable} as long as their left upper side is not a cut edge
  of $P$. Let $R(P)$ be the set of removable cells of $P$.
\item If all fin cells are marked as removable then removing these
  cells yields a fighting fish $P_1=P\setminus R(P)$, and applying the
  construction of Case (B2) to $P_1$ gives $P$ back. Conversely any
  fish produced as in Case (B2) has all its fin cells removable.
\item Otherwise let $c$ be the first fin cell which is not
  removable. Upon cutting the left upper side $e$ of $c$, two
  components are obtained: let $P_2$ be the component containing $c$
  and let $\bar P_1$ be the other component, which contains by
  construction all the removable cells of $P$. Using the incremental
  construction of fighting fish one easily check that $P_1=\bar
  P_1\setminus R(P)$ is a (possibly empty) fighting fish, and $P_2$ is
  a non-empty fighting fish.
\begin{itemize}
\item If $P_1$ is empty then applying the construction of Case (B2) to
  $P_1$ yields $P$ back, and conversely all fish produced as in Case
  (B2) have a decomposition with $P_1$ empty.
\item Otherwise the edge $e$ corresponds to a right lower side $\bar
  e_1$ on the fin of $\bar P_1$, or equivalently to an edge $e_1$ of
  the fin of $P_1$: if $\bar e_1$ is a side of a removable cell of $P$
  then $e_1$ is the right upper edge of this cell, which is a right
  lower edge on the fin of $P_1$ (this corresponds to Case (C3));
  otherwise $e_1=e$ is a right lower edge on the fin of $P_1$ (this
  corresponds to Case (C1) or (C2) depending whether $e_1$ is the
  rightmost edge on the fin of $P_1$ or not.
\end{itemize}
\end{itemize}
\end{proof}

\begin{proof}[Proof of the corollary] The wasp-waist decomposition of Theorem~\ref{thm:wasp-waist}
is readily translated into the following functional equation:
\begin{align*}
P(u) &= tu + tub P(u) + tua P(u) + tuab P(u)^2 + \\
&  ytabP(u) \sum_{P_1} t^{\mathrm{size} (P_1)-1} y^{\mathrm{tails} (P_1)-1} a^{\mathrm{rsize} (P_1)-1} 
b^{\mathrm{lsize} (P_1)-1} \left( u+\ldots +u^{\mathrm{fin} (P_1)-1} \right )\\
&= tu(1+aP(u))(1+bP(u))+ytabuP(u)\frac{P(1)-P(u)}{1-u},
\end{align*}
where the only difficult point is to observe that given a pair
$(P_1,P_2)$ of non empty fighting fish with $\mathrm{fin}(P_1)=k+1$,
Cases~(C2) and~(C3) together produce $k$ fighting fish with fin size
$2$, $3$, \dots, $k+1$ respectively.
\end{proof}
\section{Proof of Theorem~\ref{thm:gfP1}}\label{ap:gf}
Let us resume the proof from the system formed by Equations~\eqref{eqn:eq1U}, \eqref{eqn:eq2U} and~\eqref{eqn:eq3U}.

Comparing Equation~\eqref{eqn:eq1U} and Equation~\eqref{eqn:eq3U}
multiplied by $U$ we immediately deduce the simpler relation
\begin{equation}\label{eqn:eq4U}
P(U)=tU^2(1+aP(U))(1+bP(U)).
\end{equation}
Now comparing Equation~\eqref{eqn:eq1U} and Equation~\eqref{eqn:eq2U}
multiplied by $P(U)$ yields, up to canceling a factor $tU$,
\begin{equation*}
(U-1)(1+(a+b)P(U)+abP(U)^2)=(U-1)P(U)(a+b+2abP(U))+yabP(U)^2,
\end{equation*}
that is
\begin{equation}\label{eqn:eq5U}
U=1+y\frac{abP(U)^2}{1-abP(U)^2}
\end{equation}
In view of Equations~\eqref{eqn:eq4U} and~\eqref{eqn:eq5U}, $P(U)$ is the unique formal power series solution of the equation:
\begin{equation}\label{eqn:eq6U_2}
P(U)=t\left(1+y\frac{abP(U)^2}{1-abP(U)^2}\right)^2(1+aP(U))(1+bP(U)).
\end{equation}
In other terms $P(U)=B$ as defined in Theorem~\ref{thm:gfP1}.  Now using
Equation~\eqref{eqn:eq5U} to eliminate $U$ in
Equation~\eqref{eqn:eq1U}, and canceling a factor $yabP(U)$ we have:
\begin{equation*}
\frac{P(U)^2}{1-abP(U)^2}
=tU\frac{P(U)}{1-abP(U)^2}
(1+aP(U))(1+bP(U))+tU(P(U)-P(1)).
\end{equation*}
Using Equation~\eqref{eqn:eq4U} to expand a factor $P(U)$ in the left
hand side, and canceling a factor $tU$, this equation can be
rewritten as:
\begin{equation}\label{eqn:eq}
\frac{P(U)U(1+aP(U))(1+bP(U))}{1-abP(U)^2}
=\frac{P(U)}{1-abP(U)^2}
(1+aP(U))(1+bP(U))+(P(U)-P(1)).
\end{equation}
In other words:
\begin{equation*}
P(U)-P(1)=(U-1)
\frac{P(U)(1+aP(U))(1+bP(U))}{1-abP(U)^2}
\end{equation*}
and using again Equation~\eqref{eqn:eq4U},
\begin{equation}\label{eqn:eq7U}
P(U)-P(1)=
\frac{yabP(U)^3(1+aP(U))(1+bP(U))}{(1-abP(U)^2)^2}.
\end{equation}
Finally
\begin{equation}\label{eqn:eq8U_2}
P(1)=P(U)-
\frac{yabP(U)^3(1+aP(U))(1+bP(U))}{(1-abP(U)^2)^2},
\end{equation}
which concludes the proof of the theorem using $P(U)=B$.

\section{Proof of the bivariate formulas}\label{ap:lagrange}
\textbf{In this proof we implicitly set $y=1$ in all series.}
Theorems~\ref{thm:count} and~\ref{thm:marked} can be derived by
bivariate Lagrange inversion on the expression of $P(1)$ in terms of
the series
\begin{align*}
\bar R=\frac{aB(1+bB)}{1-abB^2}\qquad\textrm{and}\qquad
\bar S=\frac{bB(1+aB)}{1-abB^2}.
\end{align*}
Indeed 
\begin{equation}\label{eq:B-RS}
B=t\frac{(1+aB)(1+bB)}{(1-abS^2)^2}=t(1+\bar R)(1+\bar S),
\end{equation}
so that $\bar R$ and $\bar S$ satisfy
\begin{align}\label{eq:bivarlarg}
\begin{cases}
\bar R&=ta(1+\bar R)(1+\bar S)^2\\
\bar S&=tb(1+\bar R)^2(1+\bar S).
\end{cases}
\end{align}
Indeed Equation~\eqref{eqn:eq8U} then rewrites as
\begin{align}\label{eq:P1-RS}
P(1)&=B-abB^3\frac{(1+aB)(1+bB)}{(1-abB^2)^2}
=t(1+\bar R)(1+\bar S)(1-\bar R\bar S).
\end{align}
Given a system $\{A_1=a_1\Phi_1(A_1,A_2),A_2=a_2\Phi_2(A_1,A_2)\}$ the bivariate Lagrange inversion theorem states that for any function $F(x_1,x_2)$, 
\begin{align*}
[a_1^{n_1}a_2^{n_2}]F(A_1,A_2)=\frac1{n_1n_2}
[x_1^{n_1-1}x_2^{n_2-1}]&\left(
\frac{\partial^2 F(x_1,x_2)}{\partial x_1\partial x_2}\Phi_1(x_1,x_2)^{n_1}
\Phi_2(x_1,x_2)^{n_2}\right.\\
&+\frac{\partial F(x_1,x_2)}{\partial x_1}\frac{\partial \Phi_1(x_1,x_2)^{n_1}}{\partial x_2}\Phi_2(x_1,x_2)^{n_2}\\
&\left.+\frac{\partial F(x_1,x_2)}{\partial x_2}\frac{\partial \Phi_2(x_1,x_2)^{n_2}}{\partial x_1}\Phi_1(x_1,x_2)^{n_1}\right)
\end{align*}
In other words, \begin{equation}\nonumber
[a_1^{n_1}a_2^{n_2}]F(A_1,A_2)=\frac1{n_1n_2}
[x_1^{n_1-1}x_2^{n_2-1}]\Phi_1(x_1,x_2)^{n_1}
\Phi_2(x_1,x_2)^{n_2} H,\mbox{ where}\end{equation} \begin{equation}\nonumber
H=\frac{\partial^2 F(x_1,x_2)}{\partial x_1\partial x_2}
+{n_1}\frac{\partial F(x_1,x_2)}{\partial x_1}\frac{\partial \Phi_1(x_1,x_2)}{\partial x_2}\frac{1}{\Phi_1(x_1,x_2)}+{n_2}\frac{\partial F(x_1,x_2)}{\partial x_2}\frac{\partial \Phi_2(x_1,x_2)}{\partial x_1}\frac{1}{\Phi_2(x_1,x_2)}.\end{equation}

Setting $t=1$ and applying the bivariate Lagrange inversion formula to the function
$B(\bar R,\bar S)$ in Equation~\eqref{eq:B-RS}, where $\bar R=a\Phi_1(\bar R,\bar
S)$ and $\bar S=b\Phi_2(\bar R,\bar S)$ as defined in system~\eqref{eq:bivarlarg}, yields
\begin{align*}
[a^{i-1}b^{j-1}]B&=\frac{1}{(i-1)(j-1)}
[\bar R^{i-2}\bar S^{j-2}]\left((1+\bar R)^{i+2j-3}(1+\bar
S)^{2i+j-3}(2i+2j-3)\right)\\
&=\frac{(2i+2j-3)}{(i-1)(j-1)}\binom{2j+i-3}{i-2}\binom{2i+j-3}{j-2}\\
&=\frac{(2i+2j-3)}{(2i-1)(2j-1)}\binom{2j+i-3}{i-1}\binom{2i+j-3}{j-1}.
\end{align*}
This proves Theorem~\ref{thm:marked}.
Now, apply the bivariate Lagrange inversion formula to the function $P(1)$ in
Equation~\eqref{eq:P1-RS}: 
it holds
\begin{align*}
[a^{i-1}b^{j-1}]P(1)=\frac{1}{(i-1)(j-1)}
[\bar R^{i-2}\bar S^{j-2}]\left((1+\bar R)^{i+2j-3}(1+\bar
S)^{2i+j-3}\right.&(-4+4\bar R\bar S\\
&+2i(1-2\bar R\bar S-\bar S)\\
&\left.+2j(1-2\bar R\bar S-\bar R))\right).
\end{align*}
By extracting coefficients of $\bar R^{i-2}\bar S^{j-2}$ yields
\begin{align*}
[a^{i-1}b^{j-1}]P(1)=\frac{1}{(i-1)(j-1)}&\left(2(i+j-2)\binom{2j+i-3}{i-2}\binom{2i+j-3}{j-2}\right.\\
&-4(i+j-1)\binom{2j+i-3}{i-3}\binom{2i+j-3}{j-3}\\
&\left.-2j\binom{2j+i-3}{i-3}\binom{2i+j-3}{j-2}-2i\binom{2j+i-3}{i-2}\binom{2j+i-3}{j-3}\right).
\end{align*}
Manipulating and summing all the binomial coefficients it results
\begin{equation}\nonumber
[a^{i-1}b^{j-1}]P(1)=\frac1{ij}\binom{2j+i-2}{i-1}\binom{2i+j-2}{j-1}.
\end{equation}

\section{Bijective interpretations}\label{ap:bijections}
\subsection{A bijective proof of the relation $P^>=P(U)$}
\begin{pro}
There is a bijection between 
\begin{itemize}
\item fighting fish with a marked tail having $i+1$ left lower free edges
  and $j+1$ right lower free edges,
\item and pairs $(P,S)$ where $P$ is a fighting fish with fin size
  $k+1$ and $S$ is a $k$-uple $(U_1,\ldots,U_k)$ of sequences
  $U_i=(V_{i,1},\ldots,V_{i,j_i})$ of fish that are marked on an upper
  flat point or a branch point, such that the total number of left
  lower free edges and right lower free edges are respectively $i+1$ and $j+1$.
\end{itemize}
\end{pro}
The bijection is illustrated by Figure~\ref{fig:compose}.
\begin{figure}
\centering
\includegraphics[width=0.8\textwidth]{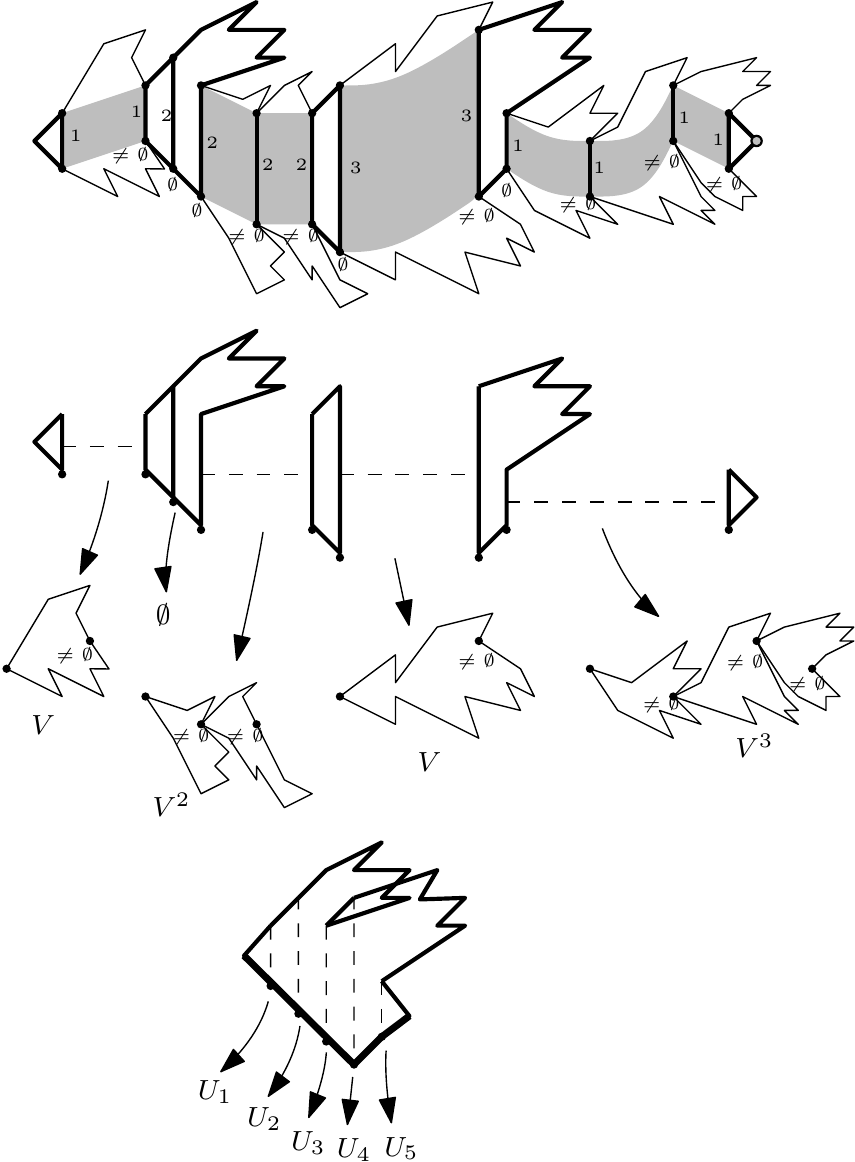}
\caption{The bijective interpretation of $P^>=P(U)$.}
\label{fig:compose}
\end{figure}
\begin{proof}[Sketch of proof] 
Given a pair $(P,S)$ as above, mark the first tail of $P$, then slice
$P$ above each inner point of its fin, cut each $V_{i,j}$ at its nose
and marked point and inflate it vertically to match the width of $P$
above the $i$th point $x_i$ of its fin, and insert the inflated
sequence $U_i$ between the slice before and after $x_i$. This produce
a fighting fish $P'$ with a marked tail.  The fact that the marked
point of each $V_{i,j}$ is an upper flat point or a branch point ensures
that the resulting point in $P'$ is a branch point. 

Conversely a fish $P'$ with a marked tail can be decomposed upon
traveling along the \emph{spine} connecting the tail to the nose and
cutting the fish in slices: Starting from the tail, 
\begin{enumerate}
\item travel to the left along the current lower free edge to reach a point $x$
\item if $x$ is a flat point then a new slice of $P$ is obtained above
  the edge that has been traversed; resume at step 1;
\item otherwise $x$ is a branch point, then let $\ell$ denote the
  length of the shortest vertical cut above $x$ that separates the
  nose and the tail:
\begin{enumerate}
\item travel along the spine until the length of the vertical cut
  returns to the the value $\ell$ for the first time, above a new
  lower point $x$: the resulting slice gives the next factor $V_{i,j}$
  in the decomposition;
\item if the new point $x$ is again a branch point then resume at the
  previous step,
\item otherwise $x$ is a lower flat point and a new factor $U_i$ has been completed; resume at step 1.
\end{enumerate} 
\end{enumerate}
The proof that the two constructions above are inverse one of the
other is omitted.
\end{proof}

\subsection{A bijective proof of the relation $V=ytab(\Delta P)(U)$}
Let $(\Delta P)(u)=u\frac{P(u)-P(1)}{u-1}=P(u^k\to u+\ldots+u^k)$.
Then $(\Delta P)(u)$ is the generating series of fighting fish with
a marked edge on the fin, where $u$ marks the distance between the
nose and the endpoint of the marked edge.

\begin{pro}
There is a bijection between 
\begin{itemize}
\item fighting fish with a marked branch or flat lower point having
  $i+1$ left lower free edges and $j+1$ right lower free edges,
\item and pairs $(P,S)$ where $P$ is a fighting fish with a marked
  edge on the fin at distance $k$ from the nose and $S$ is a $k$-uple
  $(U_1,\ldots,U_k)$ of sequences $U_i=(V_{i,1},\ldots,V_{i,j_i})$ of
  fish that are marked on an upper flat point or a branch point, such
  that the total number of left lower free edges and right lower free edges are
  respectively $i+1$ and $j+1$.
\end{itemize}
\end{pro}
\begin{proof}[Sketch of proof]
This bijection is based on the same decomposition as the previous one.
The only difference is that the substitution of $U$ factors is not
made along the whole fin of the fish: as a result the marked point is
a flat lower point or a branch point instead of being a tail.
\end{proof}
\subsection{A bijective proof of the relation $ytabP^<=V^2$}
\begin{pro}
There is a bijection between
\begin{itemize}
\item fighting fish a marked branch point having $i+1$ left lower free edges and $j+1$ right lower free edges,
\item and pairs of fighting fish $(P_1,P_2)$ marked on a branch or lower flat point, such that the total number of left lower free edges and right lower free edges are
  respectively $i+1$ and $j+1$.
\end{itemize}
\end{pro}
The bijection is illustrated by Figure~\ref{fig:branch}.
\begin{figure}
\centering
\includegraphics[width=0.8\textwidth]{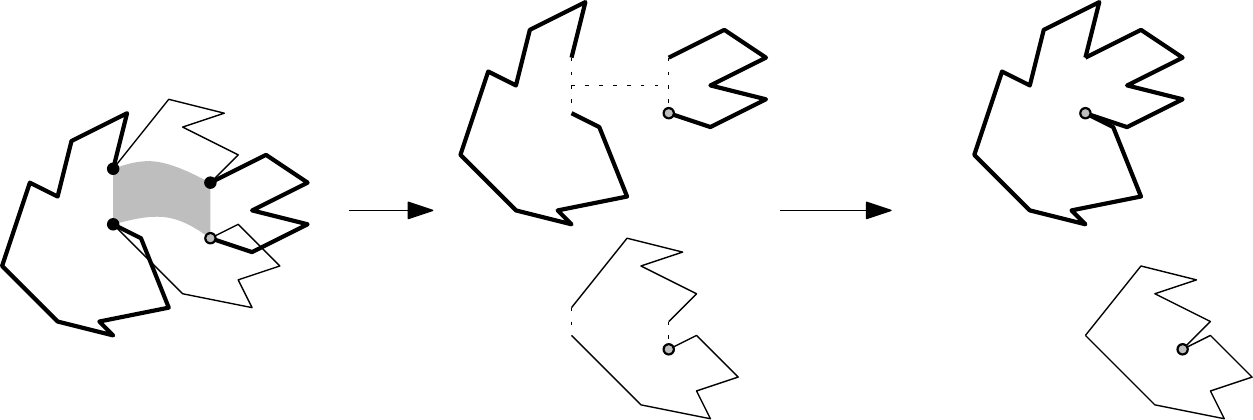}
\caption{The bijective interpretation of $ytabP^<=V^2=(ytab(P^-+P^<))^2$.}
\label{fig:branch}
\end{figure}
\begin{proof}
Let $P$ be a fighting fish with a marked branch point $x$, and let
$\ell$ denote the length of the shortest vertical cut segment above
$x$. Then the decomposition is obtained by cutting $P$ at this vertical
cut segment and at the first vertical cut segment of the same length
that is found along the spine from $x$ toward the nose of $P$. 
\end{proof}
\subsection{A bijective proof of the relation $P^>=tU^2(1+aP^>)(1+bP^>)$}
This relation follows from the wasp-waist decomposition of fighting
fish upon substituting $u=U$: the relation $P^>=P(U)$ and
$ytabU(\Delta P)(U)=V$ indeed immediately leads to
$P^>=tU(1+aP^>)(1+bP^>)+VP^>$, from which one concludes by iteration.

\end{document}